\documentclass[11pt]{amsart}
\usepackage{amsmath,amsfonts,amssymb,latexsym}
\usepackage{hyperref}
\hypersetup{
breaklinks=true, 
urlcolor= blue, 
}
\newtheorem{theorem}{Theorem}[section]
\newtheorem{proposition}[theorem]{Proposition}

\def\R{{\mathbb R}}

\newcommand{\xR}{{]}{-\infty},+\infty]}
\newcommand{\Rex}{\xR}

\newcommand{\ps}{\smallbreak}

\newcommand{\lsc}{lower semicontinuous}

\newcommand{\del}{\partial}

\newcommand{\dom} {{\rm dom} \kern.15em}
\newcommand{\cl} {{\rm cl} \kern.15em}
\newcommand{\tq}{:}
\newcommand{\la}{\langle}
\newcommand{\ra}{\rangle}

\newcommand{\eps}{\varepsilon}
\newcommand{\tos}{\rightrightarrows}
\newcommand{\bx}{\bar{x}}
\newcommand{\xb}{\bar{x}}
\newcommand{\xt}{x_\eps}
\newcommand{\yt}{y_\eps}

\begin{document}
\thispagestyle{empty}

\title
{Br{\o}ndsted-Rockafellar property of subdifferentials of prox-bounded functions}

\author{Marc Lassonde}

\address{Universit\'e des Antilles et de la Guyane,
  97159 Pointe \`a Pitre, France}

\email{marc.lassonde@univ-ag.fr}

\begin{abstract}
We provide a new proof that the subdifferential of a
proper lower semicontinuous convex function on a Banach space
is maximal monotone by adapting the pattern commonly used in the Hilbert setting.
We then extend the arguments to show more precisely that subdifferentials of
proper \lsc\ prox-bounded functions possess the Br{\o}ndsted-Rockafellar property.
\end{abstract}

\date{June 23, 2013}



\maketitle

Maximal monotonicity is a key property of subdifferentials of
proper \lsc\ convex functions on a Banach space.
The first proof of this fact in the special case of a Hilbert space
was given by Moreau \cite{Mor65} in 1965 as an application of his theory of prox functions
(and motivated by Minty's characterization of maximal monotonicity in a Hilbert space
\cite{Min62}).
Later on, Brezis \cite{Brebook73} proposed a more direct argument to demonstrate
the property in the Hilbert case. 
The general situation of an arbitrary Banach space
was handled by Rockafellar \cite{Roc70b} in 1970.
Since then, various proofs were given till recently,
see \cite{MS08,Sim09} and the references therein.
In this paper, we present yet another proof and discuss its extensions.

A standard method for proving the maximality of a monotone operator $T:X\tos X^*$
in a Hilbert space $X$, or more generally in a reflexive Banach space $X$, is to show that
$T+J$ (where $J:X\tos X^*$ is the duality mapping) is onto. This argument does
not work in a non-reflexive space since $T+J$ is not onto in general.
However, $T+J$ may have a dense range and
it is easily seen that, for any operator $T$ on any Banach space $X$,
if for every $x\in X$, $T(x+.)+J$ has a dense range, then $T$ contains
all its monotonically related points.
So, to prove the maximal monotonicity of the subdifferential $\del f$ of a convex function $f$, it suffices
to show that $\del f(x+.)+J$ has a dense range. It turns out that this
readily follows from Br{\o}ndsted-Rockafellar's approximation theorem
combined with the subdifferential sum rule. The details of this argument are given in the
first section.

In the second section, we come back on the two steps of the proof.
We show that the dense range property of the operator $T+\lambda J$ is in fact enjoyed by
any subdifferential $T=\del f$ of any (possibly non-convex) prox-bounded function.
And we show that such a dense range property of the operator $T+\lambda J$ actually
forces $T$ to possess a property reminiscent of Br{\o}ndsted-Rockafellar's, roughly:
for every $\eps\ge 0$, the points in $X\times X^*$ which are $\eps$-monotonically related to the graph of $T$
belong to a ($\sqrt{\lambda^{-1}\eps}\times \sqrt{\lambda\eps}$)-entourage of the graph of $T$.
Combining these two results then establishes that subdifferentials of prox-bounded functions
satisfy a Br{\o}ndsted-Rockafellar type property.
This extends and simplifies known results for convex functions.


\section{Maximal monotonicity of the convex subdifferential}
In the following, $X$ is a real Banach space with unit ball $B_X$,
$X^*$ is its topological dual with unit ball $B_{X^*}$,
and $\la .,. \ra$ is the duality pairing.
The closed ball centered at a point $x\in X$ with radius $\eps>0$ is denoted by
$B(x,\eps):=\{ y\in X\tq \|x-y\|\le \eps\}$.
The closure of a subset $A\subset X$ is written $\overline{A}$ or $\cl(A)$ as well.
All the functions $f : X\to\Rex$ are assumed to be lower semicontinuous
and \textit{proper}, which means that
the set  $\dom f:=\{x\in X\tq f(x)<\infty\}$ is nonempty.
Set-valued operators $T:X\rightrightarrows X^*$
are identified with their graph $T\subset X\times X^*$,
so $x^*\in Tx$ is equally written as $(x,x^*)\in T$.
The \textit{domain} and the \textit{range} of $T$ are the sets respectively given by
$D(T)=\{ x\in X : Tx\ne \emptyset\,\}$ and
$R(T)=\{ x^*\in X^* : \exists x\in X : x^*\in Tx \}$.
\smallbreak
A set-valued operator $T:X\tos X^*$,
or graph $T\subset X\times X^*$,
is said to be {\it monotone} provided
$\langle y^*-x^*,y-x\rangle \ge 0$ for every $(x,x^*)\in T$ and $(y,y^*)\in T$,
and {\it maximal monotone} provided it is monotone and not properly contained in any other
monotone operator.
\smallbreak
The \textit{subdifferential} of a proper convex function $f:X\to\Rex$
is the set-valued operator
$\partial f:X\tos X^*$ given by
$$
\partial f(x):= \bigl\{x^*\in X^*: \langle x^*, y-x \rangle +f(x)\le f(y),\,\forall y\in X\bigr\},
$$
and the {\it duality operator} from $X$ into $X^*$ is the set-valued operator
$J:X\tos X^*$ given by
$$
J(x):= \bigl\{x^*\in X^*: \langle x^*, x \rangle = \|x\|^2 = \|x^*\|^2\bigr\}.
$$
It is easily verified that $J(x)=\partial j(x)$ where
$j(x)=(1/2)\|x\|^2$.
\smallbreak
We can now state and prove the fundamental theorem of Rockafellar \cite{Roc70b}:

\begin{theorem}\label{Roc}
Let $X$ be a Banach space. The subdifferential $\partial f$ of
any proper convex \lsc\ function $f:X\to\Rex$ is maximal monotone.
\end{theorem}
\begin{proof}
Let us first recall the simple proof proposed in \cite{Brebook73}
for the special case when $X$ is a Hilbert space.
By the Hahn-Banach theorem, $f\ge \ell +\alpha$ for some $\ell\in X^*$ and $\alpha\in \R$,
and since $j+\ell$ is coercive
($j(x)+\ell(x)=(1/2)\|x\|^2+\ell(x)\to +\infty$ as $\|x\|\to +\infty$),
it follows that $f+j$ is also coercive.
Hence $f+j$ attains its minimum at some $\xb\in X$,
that is, $0\in\del(f+j)(\xb)$.
Since $\del j=\nabla j=I$ (the identity mapping on $X$), we readily get
$0\in (\del f+ I)(\xb)$, so $0\in R(\del f+I).$
We conclude that $X^*= R(\del f+I),$
which is easily seen to imply that $\del f$ is maximal monotone
(this is the elementary part in Minty's characterization of maximal monotonicity \cite{Min62}).
\ps
Let us now adapt the above argument in the case where $X$ is an arbitrary Banach space.
First, we claim that
\begin{equation}\label{claim}
 0\in\overline{R(\del f+J)}.
\end{equation}
Indeed,
since $f\ge \ell +\alpha$ for some $\ell\in X^*$ and $\alpha\in \R$,
and since $j+\ell$ is bounded from below,
we derive that $f+j$ is bounded from below.
Let $\eps>0$ and $y_\eps\in \dom f$ such that
\begin{equation*}\label{Roc2}
(f+j)(y_\eps)\le (f+j)(y) +\eps^2,\  \forall y\in X.
\end{equation*}
By Br{\o}ndsted-Rockafellar's approximation theorem \cite{BR65}, there exist
$x^*_\eps\in X^*$ with $\|x^*_\eps\|\le \eps$ and $z_\eps\in X$
(with $\|z_\eps-y_\eps\|\le\eps$)
such that
$x^*_\eps\in\del(f+j)(z_\eps)$.
From Rockafellar's subdifferential sum rule, we infer that
$x^*_\eps\in \del f(z_\eps)+J(z_\eps)$.
So, for any $\eps>0$ there exists $x^*_\eps\in R(\del f+J)$
with $\|x^*_\eps\|\le\eps$.
This proves the claim.
\medbreak
Let us now prove the maximal monotonicity of $\partial f$.
Let $(x,x^*)\in X\times X^*$ such that
\begin{equation}\label{Roc1}
\langle y^*-x^*,y-x\rangle \ge 0,\  \forall (y,y^*)\in \partial f.
\end{equation}
We have to show that $x^*\in \del f(x)$.
Applying (\ref{claim}) to the convex \lsc\ function $f(x+.)-x^*$,
we get that
$$x^*\in \overline{R(\partial f(x+.)+J)}.$$
Thus, there are $(x_n^*)\subset X^*$ with $x_n^*\to x^*$
and $(h_n)\subset X$ such that $x_n^*\in \del f(x+h_n)+J(h_n)$.
Let $(y_n^*)\subset X^*$ such that
$$y_n^*\in \del f(x+h_n)\quad\mbox{and}\quad x_n^*-y_n^*\in J(h_n).$$
By definition of $J$, we have
\begin{equation}\label{Roc3}
\la x_n^*-y_n^*, h_n\ra=\|x_n^*-y_n^*\|^2=\|h_n\|^2.
\end{equation}
From (\ref{Roc1}) and $y_n^*\in \del f(x+h_n)$, we get
$\la x^*-y_n^*, x+h_n-x\ra\le 0$, so
$$
\|h_n\|^2=\la x_n^*-x^*, h_n\ra+\la x^*-y_n^*, x+h_n-x\ra
\le \la x_n^*-x^*, h_n\ra
\le \|x_n^*-x^*\|\|h_n\|.
$$
It follows that $h_n\to 0$, so, by (\ref{Roc3}),
$\|x_n^*-y_n^*\|\to 0$, hence $y_n^*\to x^*$.
Since $\del f$ has a norm$\times$norm-closed graph and $y_n^*\in \del f(x+h_n)$,
we conclude that $x^*\in \del f(x)$.
This completes the proof.
\end{proof}
\section{Extension}
In this section,
we strengthen the maximal monotonicity property of the subdifferential and 
we show how this property can be extended beyond the convex case. 
We are supposed to be given a subdifferential $\del$ that
associates an operator $\del f:X\rightrightarrows X^\ast$ to each function $f$ on $X$
so that it coincides with the convex subdifferential when $f$ is convex.
The two main tools used in the convex situation,
Br{\o}ndsted-Rockafellar's approximation theorem
and Rockafellar's subdifferential sum rule, are respectively 
replaced by Ekeland's variational principle \cite{Eke74} and the
subdifferential separation principle. They read as follows:
\ps
\textit{Variational Principle}.
For any \lsc\ function $f$ on $X$, $\bx\in\dom f$ and $\eps>0$ such that
$
f(\xb)\le \inf f(X) +\eps,
$
and for any $\lambda>0$, there exists $x_\lambda\in X$ such that
$\|x_\lambda-\xb\|\le\lambda$, $f(x_\lambda)\le f(\xb)$ and
the function $x\mapsto f(x)+(\eps/\lambda)\|x-x_\lambda\|$
attains its minimum at $x_\lambda$.
\ps
\textit{Separation Principle}.
For any \lsc\ functions $f,\varphi$ on $X$ with $\varphi$ convex Lipschitz near $\xb\in\dom f \cap\dom \varphi$,
if $f+\varphi$ admits a local minimum at $\xb$, then
$0\in \del f(\xb)+ \del \varphi(\xb).$
\ps
\textit{Examples}.
The Clarke subdifferential, the Michel-Penot subdifferential,
the Ioffe subdifferential satisfy the Separation Principle in any Banach space.
The limiting versions of the elementary subdifferentials
(proximal, Fr\'echet, Hadamard, G\^ateaux, \ldots) satisfy the Separation Principle
in appropriate Banach spaces (see, e.g., \cite{Iof12,JL12} and the references therein).
\ps
For a function $f$ on $X$, we let
\begin{equation}\label{dom}
\dom f^*=\{x^*\in X^*:\inf (f-x^*)(X)>-\infty\}.
\end{equation}
(This is the domain of the conjugate function 
$f^*:x^*\in X^*\mapsto f^*(x^*):=\sup (x^*-f)$.)
\ps
Merging the above two principles yields: 

\begin{proposition}\label{BPnonconvex}
Let $X$ be a Banach space, $f:X\to\Rex$ be proper \lsc\ and $\varphi:X\to\R$
be convex locally Lipschitz.
Then, $\dom (f+\varphi)^*\subset \cl({R(\del f+\del \varphi)})$.
\end{proposition}
\begin{proof}
Let $x^*\in \dom (f+\varphi)^*$ and let $\eps>0$.
There is a point $\bx\in X$ such that
$$
(f+\varphi-x^*)(\xb)\le \inf (f+\varphi-x^*)(X) +\eps^2,
$$
so, by Ekeland's variational principle,
there is a point $\xt\in X$ such that
the function $x\mapsto f(x)+\varphi(x)+\la -x^*,x\ra+\eps\|x-\xt\|$ attains
its minimum at $\xt$.
Now, applying the Separation Principle with the given $f$ and
the convex locally Lipschitz function
$\psi:x\mapsto \varphi(x)+\la -x^*,x\ra+\eps\|x-\xt\|$ we obtain
a subgradient $\xt^*\in \del f(\xt)$ such that
$-\xt^*\in \del \psi(\xt)=\del \varphi(\xt)-x^*+\eps B_{X^*}$. So,
there is $\yt^*\in \del \varphi(\xt)$ such that  $\|x^*-\yt^*-\xt^*\|\le \eps$.
Thus, for every $\eps>0$ the ball $B(x^*,\eps)$ contains
$\xt^*+\yt^*\in \del f(\xt)+\del \varphi(\xt)\subset R(\del f+\del \varphi)$.
This means that $x^*\in \cl({R(\del f+\del \varphi)})$.
\end{proof}

Said differently, Proposition \ref{BPnonconvex} asserts that 
the set of all functionals $x^*$ in $X^*$ that belong to the
range of $\del f+\del \varphi$ is dense in the set of all those
functionals $x^*$ for which $f+\varphi-x^*$ is bounded below on $X$.
The case where $\varphi=0$ and $f=\delta_C$, the indicator of a nonempty closed
convex subset $C\subset X$, reads as follows:
{\it the set $R(\del \delta_C)$ of all functionals in $X^*$ which attain their
supremum on $C$ is dense in the set of all those
functionals which are bounded above on $C$.
}
This is half part of Bishop-Phelps's theorem.
(see, e.g., \cite[Theorem 3.18 (ii), p. 48]{Phe93}).
\ps
A function $f$ on $X$ is said to be \textit{prox-bounded}
if there exists $\lambda > 0$ such that the function $f+\lambda j$
is bounded from below;
the infimum $\lambda_f$ of the set of all such $\lambda$ is called the \textit{threshold}
of prox-boundedness for $f$ \cite[Definition 1.23, p. 21]{RW98}:
$$
\lambda_f:=\{\lambda>0 : \inf (f+\lambda j)>-\infty\}.
$$
Any convex \lsc\ function $g$ is prox-bounded with threshold $\lambda_g= 0$
since any such function is bounded below by a continuous affine function.
More generally, the sum $f+g$ of a prox-bounded $f$ and of a convex \lsc\ $g$
is prox-bounded with threshold $\lambda_{f+g}\le\lambda_f$.
It follows that for every $x^*\in X^*$, $\lambda_{f+x^*}=\lambda_f$.
It is also easily verified that $f(x+.)+\lambda j$ is bounded below for any
$x\in X$ and $\lambda>\lambda_f$ (see \cite[Exercise 1.24 \& Theorem 1.25, p. 21]{RW98}).

As a result of the foregoing, if $f$ is prox-bounded with threshold $\lambda_f$,
then for every $\lambda > \lambda_f$ one has
$$\forall x\in X,\ \dom (f(x+.)+\lambda j)^*=X^*.$$
From this and Proposition \ref{BPnonconvex} we immediately get:

\begin{proposition}\label{proxbo}
Let $X$ be a Banach space and let $f:X\to\Rex$ be proper, \lsc\ and prox-bounded
with threshold $\lambda_f$.
Then, for every $\lambda>\lambda_f$ one has
$$\forall x\in X,\ \cl({R(\del f(x+.)+\lambda J)})=X^*.$$
\end{proposition}

We now elaborate upon the second argument in the proof of Theorem \ref{Roc}.
Given a set-valued operator $T:X\rightrightarrows X^*$,
or graph $T\subset X\times X^*$, and $\eps\ge 0$, we let
$$
T^\eps:=                                
\{\,(x,x^*)\in X\times X^* : \langle y^*-x^*,y-x\rangle \geq -\eps,\  \forall (y,y^*)\in T\,\}
$$
be the set of all pairs $(x,x^*)\in X\times X^*$  which are
{\it $\eps$-monotonically related} to $T$.
An operator $T$ is monotone provided $T\subset T^0$ and
maximal monotone provided $T= T^0$.
In our work \cite{JL13}, a non necessarily monotone operator $T$ was declared to be 
{\it monotone absorbing} provided $T^0\subset T$, and it was shown
that subdifferentials of arbitrary proper \lsc\ functions 
are monotone absorbing.
Here we show that subdifferentials $\del f$ of prox-bounded functions $f$ actually
satisfy a much stronger monotone absorption property,
reminiscent of Br{\o}ndsted-Rockafellar's, roughly:
for every  $\lambda>\lambda_f$ and $\eps\ge 0$,
$(\del f)^\eps$ is contained in a
($\sqrt{\lambda^{-1}\eps}\times \sqrt{\lambda\eps}$)-entourage of $\del f$.
This is based on the following observation:
\begin{proposition}\label{epsmono}
Let $T:X\tos X^*$ and $\lambda>0$.
Assume
$$
\forall x\in X,\ \cl(R(T(x+.)+\lambda J)=X^*.
$$
Then, for every $\eps\ge 0$ and $(x,x^*)\in T^\eps$ there exists a sequence
$((x_n,x^*_n))_n\subset T$ such that
$$
  \lim_n \|x-x_n\|\le\sqrt{\lambda^{-1}\eps} \mbox{ and }
\lim_n \|x^*-x^*_n\|\le\sqrt{\lambda\eps}.
$$
In short:
\begin{equation}\label{maxmonoBR}
\forall \eps\ge 0,\quad T^\eps\subset \cl\left(T+\sqrt{\lambda^{-1}\eps}B_X \times \sqrt{\lambda\eps}B_{X^*}\right).
\end{equation}
\end{proposition}
\begin{proof}
Let $\eps\ge 0$ and let $(x,x^*)\in T^\eps$. Since $T(x+.)+\lambda J$ has a dense range,
we can find a sequence $(x_n^*)\subset X^*$ with $x_n^*\to x^*$
and a sequence $(y_n)\subset X$ such that $x_n^*\in T(x+y_n)+\lambda Jy_n$.
Let $(y_n^*)\subset X^*$ such that
$$y_n^*\in T(x+y_n)\quad\mbox{and}\quad x_n^*-y_n^*\in \lambda Jy_n.$$
By definition of $J$, we have
\begin{equation}\label{dudu2}
\lambda^{-1}\la x_n^*-y_n^*, y_n\ra=\|\lambda^{-1}(x_n^*-y_n^*)\|^2=\|y_n\|^2.
\end{equation}
Since $x^*\in T^\eps x$ and $y_n^*\in T(x+y_n)$, we derive that
$\la x^*-y_n^*, y_n\ra\le \eps$, hence
\begin{equation*}\label{toto}
\lambda\|y_n\|^2=\la x_n^*-x^*, y_n\ra+\la x^*-y_n^*, y_n\ra
\le \la x_n^*-x^*, y_n\ra+\eps
\le \|x_n^*-x^*\|\|y_n\|+\eps.
\end{equation*}
Therefore, $\lambda\|y_n\|^2-\|x_n^*-x^*\|\|y_n\|-\eps\le 0$, so we must have
\begin{equation}\label{toto2}
\|y_n\|\le (\|x_n^*-x^*\|+\sqrt{\|x_n^*-x^*\|^2+4\eps\lambda})/2\lambda.
\end{equation}
From (\ref{toto2}) we derive that
$$
\limsup_n \|y_n\|\le \sqrt{\lambda^{-1}\eps},
$$
so, by (\ref{dudu2}),
$$
\limsup_n \|x_n^*-y_n^*\|=\limsup_n \lambda\|y_n\|\le \sqrt{\lambda\eps}.
$$
In conclusion we have
\begin{equation*}\label{toto3}
(x+y_n,y_n^*)\in T,\quad
\limsup_n \|x-(x+y_n)\|\le \sqrt{\lambda^{-1}\eps},\quad
\limsup_n \|x^*-y_n^*\|\le \sqrt{\lambda\eps},
\end{equation*}
and without loss of generality we can replace $\limsup_n$ by $\lim_n$.
This completes the proof.
\end{proof}

For a monotone operator $T$, Property (\ref{maxmonoBR}) for every $\lambda>0$ amounts
to the so-called maximal monotonicity of Br{\o}ndsted-Rockafellar type
studied in Simons \cite{Sim99,Simbook08} (see also \cite{MS08b}).
In reflexive Banach spaces, all maximal monotone operators satisfy this property.
In non-reflexive spaces, not all maximal monotone operators satisfy it,
see, e.g. \cite[Example 47.9]{Simbook08}, but subdifferentials of
proper \lsc\ convex functions do, see \cite[Theorem 13.1 (b)]{Sim99}
or \cite[Theorem 48.4 (g)]{Simbook08}.
Our main result below,
which follows readily by combining Propositions \ref{proxbo} and \ref{epsmono},
extends this latter result to
the class of prox-bounded non necessarily convex functions,
with a more direct proof:
\begin{theorem}\label{main}
Let $X$ be a Banach space and let $f:X\to\Rex$ be proper, \lsc\ and prox-bounded
with threshold $\lambda_f\ge 0$.
Then: 
\begin{equation}\label{maxmonoBR2}
\forall \lambda>\lambda_f,\, \forall \eps\ge 0, \quad
(\del f)^\eps \subset \cl\left(\del f+\sqrt{\lambda^{-1}\eps}B_X \times \sqrt{\lambda\eps}B_{X^*}\right).
\end{equation}
Equivalently: for all $\lambda>\lambda_f$ and $\eps\ge 0$,
$$
(x^*,x)\in(\del f)^\eps \Rightarrow \exists ((x^*_n,x_n))_n\subset
\del f : \lim_n \|x-x_n\|\le\sqrt{\lambda^{-1}\eps}\ \&\  
\lim_n \|x^*-x^*_n\|\le\sqrt{\lambda\eps}.
$$
\end{theorem}

\noindent{\it Final remark.}
Br{\o}ndsted-Rockafellar's approximation theorem can be stated as follows:
for any proper \lsc\ convex function $f$ on a Banach space $X$ one has
\begin{equation}\label{maxmonoBR3}
\forall \lambda>0,\, \forall \eps\ge 0, \quad
\del_\eps f \subset \del f+\sqrt{\lambda^{-1}\eps}B_X \times \sqrt{\lambda\eps}B_{X^*},
\end{equation}
where 
$
\partial_\eps f(x):=
\bigl\{x^*\in X^*: \langle x^*, y-x \rangle +f(x)\le f(y)+\eps,\,\forall y\in X\bigr\}.
$
Clearly, $\del_\eps f\subset (\del f)^\eps$ and the inclusion is generally proper.
So, (\ref{maxmonoBR2}) implies (\ref{maxmonoBR3}) (with a slightly relaxed right hand side).
The converse is also true but the proof is not as straightforward, as we have seen in this paper.
In fact, (\ref{maxmonoBR2}) is a more accurate inclusion, as the special case $\eps=0$
already shows.

\end{document}